\documentclass[12pt,a4paper]{article}
\usepackage{latexsym,amssymb,amsmath,amsthm}
\usepackage{enumitem}
\usepackage{xcolor}
\usepackage{graphicx}
\usepackage{subfig}

\newtheorem{theorem}{Theorem}[section]
\newtheorem{lemma}[theorem]{Lemma}

\newtheorem{conjecture}[theorem]{Conjecture}

\newcommand{\claimproofend}{\hspace*{.1mm}\hspace{\fill}}

\newcommand{\HH}{\mathcal H}

\newcommand{\size}[1]{\left|#1\right|}

\newcommand\Gr[1] {\mathrm{Gr}(#1)}
\newcommand\IG[1] {\mathrm{IG}(#1)}

\begin{document}
\title{\textbf{Hamilton cycles\\in line graphs of 3-hypergraphs}}
\author{Tom\'a\v s Kaiser$^{\:1}$\and Petr Vr\'{a}na$^{\:1}$}
\date{}

\maketitle
\begin{abstract}
  We prove that every $52$-connected line graph of a rank $3$
  hypergraph is Hamiltonian. This is the first result of this type for
  hypergraphs of bounded rank other than ordinary graphs.
\end{abstract}
\footnotetext[1]{Department of Mathematics and European Centre of
  Excellence NTIS (New Technologies for the Information Society),
  University of West Bohemia, Pilsen, Czech Republic. E-mail:
  \texttt{\{kaisert,vranap\}@kma.zcu.cz}. Supported by project GA20-09525S of
  the Czech Science Foundation.}%

\section{Introduction}
\label{sec:introduction}

We refer the reader to Section~\ref{sec:preliminaries} for any
definitions not included in this introduction.

It is easy to see that there are graphs of arbitrarily high
vertex-connectivity that do not admit a Hamilton cycle. On the other
hand, in some classes of graphs, sufficient connectivity implies
Hamiltonicity. One example is the class of planar graphs
($4$-connected planar graphs are Hamiltonian by a classic result of
Tutte~\cite{Tu}). For claw-free graphs, a conjecture of Matthews and
Sumner~\cite{MS} states that vertex-connectivity greater than or equal
to $4$ is sufficient as well.

\begin{conjecture}\label{conj:ms}
  Every $4$-connected claw-free graph is Hamiltonian.
\end{conjecture}

Conjecture~\ref{conj:ms} is open, with the following being currently
the best general result of this form.

\begin{theorem}[\cite{KV}]\label{t:kv}
  All $6$-connected claw-free graphs are Hamiltonian.
\end{theorem}

If we restrict Conjecture~\ref{conj:ms} to line graphs (which form a
subclass of the class of claw-free graphs), we obtain the following
conjecture of Thomassen~\cite{Th}. (See also~\cite{BRV} for an
extensive account of problems related to Conjectures~\ref{conj:ms}
and~\ref{conj:thomassen}.)

\begin{conjecture}\label{conj:thomassen}
  Every $4$-connected line graph is Hamiltonian.
\end{conjecture}

Ryj\'{a}\v{c}ek~\cite{R} proved that Conjectures~\ref{conj:ms}
and~\ref{conj:thomassen} are in fact equivalent. He introduced a
closure technique which shows that for any positive integer $k$, all
$k$-connected claw-free graphs are Hamiltonian if and only if all
$k$-connected line graphs are.

It is natural to ask if an analogue of Theorem~\ref{t:kv} could be
proved for $K_{1,r+1}$-free graphs, where $r\geq 3$. This is not
known, although the question has been around for quite some
time. Jackson and Wormald~\cite[p.~142]{JW} asked whether every
$(r+1)$-connected $K_{1,r+1}$-free graph is Hamiltonian, where
$r\geq 3$. Chen and Schelp~\cite{CS} noted that a conjecture of
Chv\'{a}tal~\cite{Ch} would imply that every $2r$-connected
$K_{1,r+1}$-free graph is Hamiltonian; however, this particular
conjecture of Chv\'{a}tal (`every $2$-tough graph is Hamiltonian') has
since been disproved~\cite{BBV}.

A positive result on a weaker version of the problem for $r=3$ is
established in~\cite{RVW}: $6$-connected $K_{1,4}$-free graphs which,
in addition, contain no induced copy of $K_{1,4}+e$ (the simple graph
with degree sequence $1,1,2,2,4$), are Hamiltonian.

By analogy with claw-free graphs, one might guess that the problem of
Hamiltonicity of $K_{1,r+1}$-free graphs could be reduced to the
special case of line graphs of hypergraphs of rank $r$. This may be
so, but no extension of Ryj\'{a}\v{c}ek's technique that would
accomplish this task is known. Still, line graphs of hypergraphs of
rank $r$ are a natural starting point for an investigation of
$K_{1,r+1}$-free graphs.

Even for this class of graphs, no analogue of Theorem~\ref{t:kv} is
known. The following conjecture has recently been proposed
in~\cite{GLS}:
\begin{conjecture}\label{conj:gls}
  For any $r\geq 2$, there is an integer $\phi(r)$ such that every
  $\phi(r)$-connected line graph of a rank $r$ hypergraph is
  Hamiltonian.
\end{conjecture}
A stronger form of the conjecture in~\cite{GLS} includes the statement
that $\phi(r) = 2r$ works.

Li et al.~\cite{LORV} recently found a close relation between line
graphs of rank $3$ hypergraphs and
Conjecture~\ref{conj:thomassen}. They showed the latter conjecture to
be equivalent to the statement that every $4$-connected line graph of
a rank $3$ hypergraph is \emph{Hamilton-connected} (that is, for any
two vertices $u,v$, it has a Hamilton path joining $u$ to $v$).

In this note, we establish Conjecture~\ref{conj:gls} in the first open
case, $r=3$. We use a result of DeVos et al.~\cite{DMP} on disjoint
$T$-connectors as our main tool.
\begin{theorem}\label{t:hyper}
  If $G$ is the line graph of a rank $3$ hypergraph and $G$ is
  $52$-connected, then $G$ is Hamiltonian.
\end{theorem}
The method easily extends to Hamilton-connectedness, at the price of a
slight increase in the constant (from $52$ to $54$). To keep our
notation and terminology simpler, we prove Theorem~\ref{t:hyper} just
for Hamiltonicity.

\section{Preliminaries}
\label{sec:preliminaries}

Our terminology mostly follows Bondy and Murty~\cite{BM}. Graphs may
contain parallel edges but no loops.

Given a graph $H$, we say that a graph $G$ is \emph{$H$-free} if $G$
contains no induced copy of $H$. \emph{Claw-free} is used as a synonym
for $K_{1,3}$-free.

A \emph{hypergraph} consists of a vertex set $V$ and a multiset $E$ of
\emph{hyperedges}, each of which is a nonempty subset of $V$. The
\emph{rank} of a hypergraph $\HH$ is the maximum cardinality of a
hyperedge of $\HH$. A hypergraph of rank $r$ is also referred to as an
\emph{$r$-hypergraph}.

The \emph{line graph} $L(\HH)$ of a hypergraph $\HH=(V,E)$ has $E$ as
its vertex set, with an edge linking $e$ and $f$ ($e,f\in E$) whenever
$e$ and $f$ intersect. Observe that if $\HH$ has rank $r$, then
$L(\HH)$ is $K_{1,r+1}$-free.

\section{Tools}
\label{sec:tools}

\subsection{$T$-connectors}
\label{ss:conn}

Let $T$ be an arbitrary set of vertices of a graph $G$. We say that
$T$ is \emph{$k$-edge-connected} in $G$ if for any $s_1,s_2\in T$, $G$
contains $k$ edge-disjoint paths from $s_1$ to $s_2$. By Menger's
Theorem, $T$ is $k$-edge-connected if and only if $G$ contains no
edge-cut $X$ such that $\size X < k$ and at least two components of
$G-X$ contain vertices of $T$.

Let $P$ be a path in $G$. Following~\cite{WW}, we define the operation
of \emph{short-cutting} $P$ as deleting all edges of $P$ and then
adding an edge joining the end vertices of $P$. A path in $G$ is a
\emph{$T$-path} if its end vertices belong to $T$ and none of its
other vertices are in $T$. A \emph{$T$-connector} in $G$ is the union
of a family of edge-disjoint $T$-paths in $G$ such that short-cutting
them one by one, we obtain a graph whose induced subgraph on $T$ is
connected. 
Observe that all the vertices of a $T$-connector $C$ whose degree in
$C$ is odd belong to $T$.

DeVos et al.~\cite[Theorem~1.6]{DMP} proved the following result on
edge-disjoint $T$-connectors (see also~\cite{L,WW}).
\begin{theorem}[\cite{DMP}]\label{t:dmp}
  For $k\geq 1$, if $T \subseteq V(G)$ is $(6k+6)$-edge-connected in
  $G$, then $G$ contains $k$ edge-disjoint $T$-connectors.
\end{theorem}

\subsection{Hamiltonicity of line graphs of $3$-hypergraphs}
\label{ss:ham}

A well-known result of Harary and Nash-Williams~\cite{HNW}
characterises graphs whose line graph is Hamiltonian. We use an
extension of this result to $3$-hypergraphs, given in~\cite{LORV}. (A
more general extension, valid for all hypergraphs, was found
in~\cite{GLS}.)

Let $\HH$ be a $3$-hypergraph. The \emph{incidence
  graph}\footnote{This graph is denoted by $\Gr\HH$
  in~\cite{LORV}. Since the same symbol is used in~\cite{KV} with a
  slightly different meaning, we opt for the alternative $\IG\HH$.}
$\IG\HH$ of $\HH$ is the bipartite graph with vertex set
$V(\HH)\cup E(\HH)$ and edges of the form $(v,e)$, where
$v\in V(\HH)$, $e\in E(\HH)$ and $v\in e$. The vertices of $\IG\HH$
belonging to $E(\HH)$ are called \emph{white}, the other vertices are
\emph{black}. Note that each white vertex of $\IG\HH$ has degree $2$ or
$3$.

A \emph{closed walk} $Q$ in a graph $G$ is a sequence
$v_0,e_0,v_1,\dots,e_{k-1},v_k$, such that $e_i$ is an edge of $G$
with end vertices $v_i$ and $v_{i+1}$ ($0\leq i \leq k-1$) and
$v_k = v_0$. Each of the vertices $v_i$ is said to be \emph{visited}
by $Q$ (as many times as it appears in $Q$); similarly, an edge $e_i$
is said to be \emph{traversed} by $Q$ (again with possible
multiplicity). A \emph{closed trail} is a closed walk visiting each
edge at most once.

Let $v_i$ be a vertex visited once by the above walk. The
\emph{precedessor edge} of $v_i$ is defined to be $e_{i-1}$ (with
subtraction modulo $k$). Similarly, the \emph{successor edge} of $v_i$
is $e_i$ if $i < k$, and $e_0$ otherwise.

Given an arbitrary set $W$ of vertices of degree $2$ or $3$ in $G$, a
\emph{closed $W$-quasitrail} in $G$ is a closed walk which traverses
each edge at most twice, and if an edge $e$ is traversed twice, then
it has an end vertex $w\in W$ such that $w$ is visited once and $e$ is
both the precedessor edge and the successor edge of $w$. A closed
$W$-quasitrail in $G$ is \emph{dominating} if it visits at least one
vertex in every edge of $G$.

We will use the following characterisation of $3$-hypergraphs with
Hamiltonian line graphs, which follows from~\cite[Corollary~7]{LORV}.

\begin{theorem}[\cite{LORV}]\label{t:l3h}
  Let $\HH$ be a $3$-hypergraph and let $W$ be the set of white
  vertices of its incidence graph $\IG\HH$. The line graph $L(\HH)$ of
  $\HH$ is Hamiltonian if and only if $\IG\HH$ contains a dominating
  closed $W$-quasitrail.
\end{theorem}

As remarked above, hypergraphs of arbitrary rank whose line graph is
Hamiltonian were recently characterised in~\cite{GLS}.

\section{Proof of Theorem~\ref{t:hyper}}
\label{sec:proof}

Let $L(\HH)$ be the line graph of a $3$-hypergraph $\HH$. Suppose that
$L(\HH)$ is $52$-connected; in fact, it is enough if $L(\HH)$ is
$18$-connected and its minimum degree is at least $52$. We prove that
$L(\HH)$ is Hamiltonian.

Consider the graph $\IG\HH$ and recall that every vertex of $\HH$ is a
black vertex of $\IG\HH$. Let us call such a vertex of $\IG\HH$
\emph{heavy} if its degree is at least $18$. Since the minimum degree
of $L(\HH)$ is at least $52 > 3 \cdot 17$, every hyperedge of $\HH$
contains a heavy vertex. Therefore, every white vertex of $\IG\HH$ is
adjacent to a heavy vertex. Let $T$ be the set of heavy vertices of
$\IG\HH$.

\begin{lemma}\label{l:20}
  The set $T$ is $18$-edge-connected in $\IG\HH$.
\end{lemma}
\begin{proof}
  For the sake of a contradiction, consider an edge-cut $X$ in
  $\IG\HH$ of size less than $18$ that separates two vertices
  $s_1,s_2\in T$. Each edge $e$ of $\IG\HH$ corresponds to a hyperedge
  $e'$ of $\HH$; let $X' \subseteq E(\HH)$ be the set of the (fewer
  than $18$) corresponding hyperedges for the edges in $X$. Removing
  the hyperedges in $X'$, we separate $s_1$ from $s_2$ in $\HH$. We
  claim that $X'$ is a vertex cut in $L(\HH)$; to prove this, we need
  to show that at least two components of $\HH-X'$ contain at least
  one hyperedge each.

  But this is not hard. Since $s_1$ is heavy, it is incident with at
  least $18$ hyperedges in $\HH$, and at most $17$ of these hyperedges
  can be in $X'$. Thus, at least one hyperedge $e_1$ containing $s_1$
  is a hyperedge of $\HH-X'$. Similarly, there is a hyperedge
  $e_2\notin X'$ containing $s_2$. Then, in $L(\HH)$, $e_1$ and $e_2$
  are two vertices separated by the vertex cut $X'$ of size less that
  $18$, so $L(\HH)$ is not $18$-connected contrary to the assumption.
\end{proof}

\begin{lemma}\label{l:trail}
  The graph $\IG\HH$ contains a closed trail visiting every vertex in
  $T$.
\end{lemma}
\begin{proof}
  By Lemma~\ref{l:20} and Theorem~\ref{t:dmp}, $\IG\HH$ contains two
  edge-disjoint $T$-connectors, say $A_1$ and $A_2$. It is a standard
  observation that $A_1\cup A_2$ contains a connected subgraph $C$
  with all degrees even such that $C$ covers all vertices in $T$. To
  prove it, let $B$ be the set of vertices of $A_1$ with odd degree in
  $A_1$. Then $\size B$ is even, and by the definition of
  $T$-connector, $B\subseteq T$. We partition $B$ in pairs
  arbitrarily, and join each of the pairs by a path in $A_2$. The
  symmetric difference $D$ of all these paths is a subgraph of $A_2$,
  and by a simple parity argument, the set of odd degree vertices of
  $D$ is precisely $B$. Now $A_1\cup D$ is the desired subgraph
  $C$. Since every vertex of $\IG\HH$ has even degree in $C$ and $C$
  is connected, there is a closed trail traversing precisely the edges
  in $C$.
\end{proof}

Let $R$ be a closed trail obtained from Lemma~\ref{l:trail}. We aim to
use Theorem~\ref{t:l3h} to prove that $R$ gives rise to a Hamilton
cycle in $L(\HH)$. Let $W$ be the set of white vertices of
$\IG\HH$. We now construct a closed dominating $W$-quasitrail from
$R$.

Consider a white vertex $w$ of $\IG\HH$ not visited by $R$. The vertex
$w$ is adjacent to a heavy vertex $v$, and every heavy vertex is
traversed by $R$. Let us insert in $R$ a detour to $w$ immediately
following the visit to $v$. That is, an occurrence of $v$ in $R$ will
be changed to $v,vw,w,wv,v$. Repeating this operation for each
unvisited white vertex (choosing one heavy neighbour arbitrarily if
there are more than one), we obtain a closed walk visiting each white
vertex, and therefore dominating all edges of $\IG\HH$. In fact, the
resulting walk is a closed dominating $W$-quasitrail, so
Theorem~\ref{t:l3h} implies that $L(\HH)$ is Hamiltonian. The proof is
complete.

\section*{Acknowledgment}

We thank two anonymous reviewers for their helpful comments.

\end{document}